\documentclass[11pt,  reqno]{amsart}
\usepackage{a4wide}
\usepackage{amsthm, mathtools, amssymb, amsfonts, stmaryrd, latexsym, mathrsfs, todonotes, eufrak, color}
\usepackage{graphicx}
\usepackage[english]{babel}
\usepackage[all]{xy}
\usepackage{nicefrac}
\usepackage{mparhack}
\usepackage[normalem]{ulem}
\usepackage{color}
\usepackage{blindtext}
\usepackage[inline]{enumitem}
\usepackage{fancyhdr}
\usepackage{enumitem}
\usepackage{algorithmic}
\usepackage{algorithm}
\usepackage{hyperref}
 \usepackage[subnum]{cases}

\usepackage[T1]{fontenc}
\usepackage[utf8]{inputenc}

\theoremstyle{plain}
\newtheorem{theo}{Theorem}[section]
\newtheorem{lemma}[theo]{Lemma}
\newtheorem{prop}[theo]{Proposition}
\newtheorem{corollary}[theo]{Corollary}
\newtheorem{remark}[theo]{Remark}
\newtheorem{example}[]{Example}

\theoremstyle{definition}
\newtheorem{definition}[]{Definition}

\newcommand{\F}{\mathbb{F}}

\newcommand{\Z}{\mathbb{Z}}
\newcommand{\Q}{\mathbb{Q}}

\newcommand{\lowoverline}[1]{%
  \overline{\smash{#1}\vphantom{b
 \bar \delta^n}}\vphantom{#1}%
}

\renewcommand{\epsilon}{\varepsilon}

\newcommand{\Gal}{\mathop{\mathrm{Gal}}\nolimits}
\newcommand{\Aut}{\mathop{\mathrm{Aut}}\nolimits}

\newcommand{\Hol}{\mathop{\mathrm{Hol}}\nolimits}

\newcounter{enumi_saved}

\makeatletter
\def\imod#1{\allowbreak\mkern10mu({\operator@font mod}\,\,#1)}
\makeatother

\newcommand{\nota}[1]{\marginpar{\footnotesize\fbox{\parbox{\marginparwidth}{#1}}}}

\allowdisplaybreaks

\begin{document}

\title[Module braces: relations between the additive and the multiplicative groups]
      {Module braces: relations between \\ the additive and the multiplicative groups
      }
   

\author{Ilaria Del Corso}

\address[I.~Del Corso]%
        {Dipartimento di Matematica\\
          Universit\`a di Pisa\\
          Largo Bruno Pontecorvo, 5\\
          56127 Pisa\\
          Italy}
\email{ilaria.delcorso@unipi.it}

\urladdr{http://people.dm.unipi.it/delcorso/}



\begin{abstract}
In this paper we define a class of braces, that we call module braces or $R$-braces, which are braces for which the additive group has also a module structure over a ring $R$, and for which the values of the gamma functions are automorphisms of $R$-modules.
This class of braces has already been considered in the literature in the case where the ring $R$ is a field: we generalise the definition  to any ring $R$, reinterpreting it in terms of the so-called gamma function associated to the brace, and prove that this class of braces enjoys all the natural properties one can require. We exhibit  explicit example of R-braces, and we study the splitting of a module braces in relation to the splitting of the ring $R$, generalising thereby Byott's result on the splitting of a brace with nilpotent multiplicative group as a sum 
of its Sylow subgroups.

The core of the paper is in the last two sections, in which, using methods from commutative algebra and number theory, we study the relations between the additive and the multiplicative groups of an $R$-brace showing that if a certain decomposition of the additive group is \emph{small} (in some sense which depends on $R$), then the additive and the multiplicative groups have the same number of element of each order. In some cases, this result considerably broadens  the range of applications of the results already known on this issue.
\end{abstract}

\thanks{The     author is   member   of
  INdAM---GNSAGA. The author gratefully  acknowledge support from the
  Departments of  Mathematics of  the Universities  of 
  Pisa.  The author has performed this activity in
    the framework of  the PRIN 2017, title  ``Geometric, algebraic and
    analytic methods in arithmetic''.}

\maketitle

\thispagestyle{empty}

\section{Introduction}

 A brace is an algebraic structure introduced by Rump in \cite{Rum07} as a generalisation  of a radical ring.  
On a radical ring $N$, besides the usual additive group structure, one can define another group operation, called the adjoint operation, giving on $N$ an additional group structure; the two operations on $N$  are linked by the  ``compatibility relation''
\begin{equation*}
\label{eq:comp-rel}
x \circ(y+z)
  =
  (x \circ y) -x
  +(x \circ z),
\end{equation*}
for $x,y,z\in N$.

A  \emph{(left) brace} is an abelian group $(N,+)$, together with an additional group structure $(N,\circ)$ such that the two operations are linked by the same compatibility relation as  that of the radical rings.

In the paper \cite{GV17}, the authors further generalised this concept defining a  \emph{skew (left) brace} as a brace without the requirement on the first group to be abelian.  For the skew brace $(N, +, \circ)$ the notation of braces is kept: $(N, +)$ is called the
\emph{additive} group, and $(N, \circ)$ the \emph{multiplicative} group.

With the paper \cite{GV17}, a systematic study was begun of the properties of these structures, which turn out to be very interesting since they naturally appear in many mathematical contexts.

First of all, we recall that skew braces can be described in terms of regular subgroups of the holomorph. In fact, for a fixed group $(N,+)$ the skew braces with additive group $(N,+)$ correspond to  the regular subgroups of ${\rm Hol}(N)$. As shown in \cite[Theorem 4.2]{GV17} (see also \cite{CDVS06, CDV17, CDV18}) the regular subgroup of the holomorph are in turn in one-to-one correspondence  with the  functions
$$\gamma\colon N\to\Aut(N)$$
characterised by the functional equation
$$\gamma(x+\gamma_x(y))=\gamma_x\gamma_y.$$
We call these functions \emph{gamma functions};  in the literature they are usually referred to as $\lambda$-functions.

The study of the Yang--Baxter equation has been  the motivation for the definition and for the study of skew braces. 
In fact, braces were introduced by Rump to study non-degenerate involutive
set-theoretic solutions of the Yang--Baxter equation, and in \cite{GV17} skew braces are introduced to
include the non-involutive case.

The study of  skew braces  is strictly related also to that of the Hopf--Galois structures on a finite Galois extension $L/K$.
%
    The work of Greither and Pareigis in~\cite{GP87} 
and that of Childs~\cite{Chi89} and Byott~\cite{Byo96}
showed that 
if $\Gamma=\Gal(L/K)$, to each Hopf--Galois structure on $L/K$ one can associate a group $(G, +)$ with the same cardinality as $\Gamma$ and such that   the
holomorph $\Hol(G)$ of  $(G, +)$ contains a regular  subgroup isomorphic to
$\Gamma$. Since classifying  the  regular   subgroups  of  $\Hol(G)$  is
equivalent to determining the operations  ``$\circ$'' on $G$ such that
$(G, +, \circ)$ is a (left) skew brace,
 the Hopf--Galois structures on an extension  with Galois group isomorphic to a group $\Gamma$ correspond bijectively to the skew braces $(G, +, \circ)$ with $(G,\circ)\cong\Gamma.$
This connection was first  observed by Bachiller in~\cite[\S 2]{Bac16}
and it is described in detail in the appendix to~\cite{SV18} by Byott and Vendramin (see also \cite{ST22b} for a slightly different point of view).

In recent years, these different approaches concurred to construct a
rich theory.  
A number of papers are devoted to the enumeration of the skew braces (or, equivalently, the Hopf--Galois structures) whose cardinality has a particular form, and their isomorphism classes
(\cite{Byo96, koh98, Byo04, Chi05, zen18, AB20-sqrf, AB20-pq,
  AB22, CCDC20}).  

Part of the literature is devoted to understanding  how
properties of $(G, +)$ influence those of $(G,\circ)$. and vice versa,
when $(G,+,\circ)$ is a skew brace  (\cite{FCC}, \cite{Byo13}, \cite{Byo15}, \cite{Bac16}, \cite{Tsa19}, \cite{Nas19}, \cite{TQ20},\cite{CDC22}).
Another part concerns the study of skew braces of a particular type, e.g., nilpotent skew braces (see \cite{CSV19, Smo22b}), bi-skew braces (see \cite{Chi19, Car20, ST22a}).

In this paper we address the latter two aspects. We define a class of braces, that we call \emph{module braces} or $R$-braces,
which are braces for which the additive group has also a module structure over a ring $R$ and for which the values of the gamma function are automorphisms of $R$-modules (see Definition~\ref{def:modulebrace}). The class of   $\Z$-braces is the class of braces defined by Rump.

$R$-brace has already been  considered in the literature  in the case where the ring $R$ is a field (see \cite{CCS15, CCS19, Smo22a, Smo22b}, where also some applications of this structure are provided). We generalise  the definition given there to any ring $R$, reinterpreting it in  terms of the gamma function associated to the brace. 
%
This class turns out to be very natural;  in fact, it enjoys  many of the natural properties one can expect, and their properties can be handily described via gamma functions.

In some cases,  this more general context ($R$-brace instead of $\Z$-brace) allows a  more conceptual view of classical results on braces, and  formulations more suitable for generalisations. 


This is the case of Proposition~\ref{prop:decomposition} in which we 
 study the splitting of a module braces in relation to the splitting of the ring $R$, generalising the splitting of a brace with nilpotent multiplicative group as a sum of its Sylow subgroups (see Remark~\ref{rem:decomposition}), which is the central point of \cite[Theorem~1]{Byo13} (see also \cite[Corollary~4.3]{CSV19} and \cite[Subs. 2.4]{CCDC20}).

We also study the relations between the additive and the multiplicative groups of an $R$-brace, in the case when the additive group is \emph{small}, in the sense of \cite[Theorem~1]{FCC} and  \cite[Theorem~2.5]{Bac16}. By introducing tools from commutative algebra and number theory we are able to obtain, in Theorem~\ref{teo:isomorphic} and its corollaries,  a complete generalisation of that results. 
%
\vskip.3cm

The paper is organized as follows.

In Section~\ref{sec:tools} we introduce all the notation and main results on skew braces we will need, restating
 the definitions and some basic results  in terms of the gamma function associated to a skew braces.
Although we will apply the machinery in the more specific context of braces, we develop it  in full generality,  as it could have an independent interest. Moreover, the general treatment requires hardly more effort and helps to place our results within the general theory.

In Section~\ref{sec:modulebraces} we define module braces  and prove that substructures can be naturally defined and have a good behaviour  with respect to the usual operations. 
We also construct explicit example of $R$-braces, and show that the brace associated with any radical ring with an $R$-algebra structure, over a commutative ring $R$, is always an $R$-brace.

As we already mentioned, in Proposition~\ref{prop:decomposition} we study the splitting of an $R$-brace in relation to the splitting of the (commutative) ring $R$, showing that, in some cases, to this decomposition corresponds a splitting of an $R$-brace. The condition we have on  $R$-braces when specified for $R=\Z$ correspond to the nilpotence of the multiplicative group of the brace and its splitting to that given in \cite[Theorem~1]{Byo13} as a sum of the Sylow subgroups.


Section~\ref{sec:smallrank} is the core of the paper. At the beginning  of the section we state, in Theorem~\ref{teo:isomorphic},  the following generalisation of \cite[Theorem~1]{FCC} and  \cite[Theorem~2.5]{Bac16}, which is proved in Section~\ref{sec:proof}. 
\vskip.3cm

{\bf Theorem.} 
{\sl
Let $p$ be a prime number, and let $D$ be a PID such that $p$ is a prime in $D$.
Let $(N,+,\circ)$ be a $D$-brace of order a power of $p$.

Assume 
that 
the number of cyclic factors of the decomposition of $N$ as a sum of 
\emph{indecomposable} cyclic $D$-module
is $<p-1$.

Then $(N,+)$ and $(N,\circ)$ have the same number of element of each order.
In particular, if $(N,\circ)$ is abelian, then $(N,+)\cong (N,\circ)$.
}
\vskip.3cm

For $D=\Z$ the previous theorem is  \cite[Theorem~2.5]{Bac16}.
A family of rings $D$ fulfilling the assumption of the previous theorem is given by the rings of integers of unramified extensions of $\Q_p$, the field of $p$-adic numbers. In   Section~\ref{sec:smallrank}, after introducing the necessary tools of number theory and commutative algebra, we show that these rings of integers play a fundamental role in this context, since every $R$-brace of order a power of $p$ is also naturally a module over one of such rings (see Proposition~\ref{prop:padic2}), showing that the assumption on the ring $D$ in the theorem is less restrictive than it might appear. This section also contains some corollaries that apply the theorem to specific cases.

The advantage of appealing to Theorem~\ref{teo:isomorphic}, instead of to \cite[Theorem~2.5]{Bac16}, when dealing with a $D$-braces, is that the condition of having few ciclic factors in the $D$-module decomposition can be much weaker than the condition  of having few ciclic factors in the $\Z$-module decomposition (see Lemma~\ref{lemma:rango}). Therefore, in some cases, the use of Theorem~\ref{teo:isomorphic}, instead of its classic form, considerably  broadens  the range of application.


In a forthcoming paper, we will show how   Theorem~\ref{teo:isomorphic} and its corollaries, can be used to make some progress in the study of Fuchs’ question on the group of units of a ring  (see Remark~\ref{rem:units}).

\section{Skew braces and the nilpotency series}
\label{sec:tools}

In this section, we briefly recall the basic of the skew braces language introduced in \cite{GV17} with a particular emphasis on gamma functions (see \cite{CDV17} and \cite{CCDC20})).
We also introduce the right series of a  skew left brace, following the notation given in \cite{JKVAV19}.
\subsection{Skew braces and ideals}

Let $N$ be a  set, and let ``$+$'' and ``$\circ$'' be two group operations on $N$.

 $(N, +, \circ)$ is a \emph{skew left brace} if 
the following brace axiom holds for $x,y,z \in N$
\begin{equation}
  \label{eq:left-brace-axiom}
  x \circ(y+z)
  =
  (x \circ y) -x
  +(x \circ z).
\end{equation}
Analogously, we say that $(N, +, \circ)$, is  \emph{skew right brace} if, for $x,y,z  \in N$,
\begin{equation}
  \label{eq:right-brace-axiom}
  (x+y)\circ z 
  =
  (x \circ z) -z
  +(y \circ z).
\end{equation}

In the following we will simply call \emph{skew brace} a skew left brace and we refer to $(N,+)$ as the additive group and to $(N,\circ)$ as the multiplicative group of the skew brace.

From the brace axiom it follows that the two group structures of a skew brace have the same unit element, which we will denote by 0.

%

For a skew brace $(N, +, \circ)$, equation \eqref{eq:left-brace-axiom} can be rewritten as 
\begin{equation}
  \label{eq:brace-axiom}  
  -x+x\circ(y+z)
  =-x +
  (x \circ y) - x
+
  (x \circ z) .
\end{equation}

If $(N, +, \circ)$ is a skew brace, then~\eqref{eq:brace-axiom}
and the fact that $(N, \circ)$ is a group yield that for all $x \in N$ the
maps $N \to N$ given by 
\begin{equation*}
  \gamma_x\colon y \mapsto -x+x\circ y
\end{equation*}
are automorphisms of $(N, +)$, and that $\gamma : (N, \circ)\to \Aut(N,+)$ defined by $\gamma(x)=\gamma_x$ is a homomorphism, namely
$$\gamma(x\circ y)=\gamma(x)\gamma(y)$$
for all $x,y\in N$.

Conversely, given a group $(N,+)$ we say that a map $\gamma : N \to \Aut(N)$ is a \emph{gamma function} if it verifies the following gamma functional equation 
\begin{equation}
\label{eq:GFE}
\gamma(x+\gamma_x(y))= \gamma(x)\gamma(y)
\end{equation}
for each $x,y\in N$.

For a given gamma function on $(N,+)$ we can define an operation on $N$ by 
\begin{equation}
\label{eq:op-pallino}
x \circ y =
x+ \gamma_x(y),
\end{equation}
for $x,y \in N.$ 

It is easy to check that given the gamma function $\gamma$ the operation $\circ$ as in \eqref{eq:op-pallino} gives a group structure on $N$ and that \eqref{eq:left-brace-axiom} holds, so 
 and $(N, +, \circ)$ is a skew brace. (See \cite[after Theorem~2.2]{CCDC20} for the explicit correspondence between the properties of $\gamma$ and those of $\circ$. Notice that there the case is considered of right skew braces.)

We also recall that a skew brace structure can be defined also starting from a group  $(N,+)$ with an additional binary operation usually, denoted by $\star$ such that  the following version of distributivity combined with associativity holds:
\begin{equation}
\label{eq:prop-star}
(x + y + x \star y) \star z = x \star z + y \star z + x \star (y\star z),\quad  x\star (y + z) =x\star y+ x\star z
\end{equation}
for all $ x, y, z\in N$,
with the additional condition that the operation $\circ$ defined by 
\begin{equation}
\label{eq:circ=*}
x\circ y=x+x\star y+y
\end{equation}
defines on $N$  a group structure (see \cite[Theorem~1.4]{CSV19}).


 In terms of the gamma function associated to the skew brace we have
\begin{equation}
\label{eq:*gamma}
 x\star y=\gamma_x(y)-y.
\end{equation}

Viceversa, given a (left) skew brace $(N, +, \circ)$ then 
$$x\star y=-x+x\circ y-y$$ 
verifies \eqref{eq:prop-star}. The construction of a skew brace via the $\star$ operation generalises that made by Rump for braces.

A left (right) \emph{brace} is a skew left (right) brace with abelian additive group and was introduced  in \cite{Rum07} as a generalisation of radical rings. Infact, 
if  $(N, +, \cdot)$ is a radical ring, taking $\star$ to be the product of the ring, and the  operation $\circ$ as in \eqref{eq:circ=*}, we get  that  $(N, +, \circ)$ is  two-sided brace, namely a brace for which both \eqref{eq:left-brace-axiom} and \eqref{eq:right-brace-axiom} hold. In this case we will refer to the operation $\circ$ as the adjoint operation, and to $(N,\circ)$ as to the adjoint group of the radical ring.

On the contrary, if $(N, +, \circ)$ is a two-sided brace, then defining 
$$x\cdot y=-x+ x\circ y-y$$
we have that $(N,+,\cdot)$ is a radical ring (see \cite{Rum06} or \cite[Prop.~1]{CJO14}). 
In this case $(1+N, \cdot)$ and $(N,\circ)$ are isomorphic groups via the map $1+x\to x$.

As for the gamma function associated to the brace $(N,+,\circ)$ arising from a radical ring,  we have 
$$\gamma_x(y)=-x+x\circ y=(x+1)y.$$ 
\begin{definition}
 A subset $A$ of a skew brace $(N, +, \circ)$ is called
 \begin{itemize}
 \item a \emph{subskew brace} if it is a subgroup both of $(N,+)$ and $(N,\circ)$;
\item a \emph{left ideal} if it is  $\gamma(N)$ invariant, namely $\gamma_x(A)\subseteq A$ for each $x\in N$, and it is a subgroup for any (and hence for both) group structures;
\item an \emph{ideal} if  it  is $\gamma(N)$ invariant and it is a normal subgroup of both $(N, +)$ and  $(N ,\circ)$.
\end{itemize}
\end{definition}
 Clearly any ideal is a left ideal, and any left ideal is a subskew brace. The converse does not hold in general. 
 
Notice  that the gamma function of a subskew brace is the restriction of the gamma function of the skew brace.
 
 We also recall that 
 if $I$ is an ideal of the skew brace $N$, then $N/I$  has a structure of quotient group for both the operations, so it is  a skew brace.

Two skew braces $N$ and $M$ are called \emph{isomorphic} if there exists a map $f\colon N\to M$  which is an isomorphism for both the additive and the multiplicative structures.
\begin{lemma}
\label{lemma:iso}
Let $(M, + , \circ)$ and $(N, +',\circ')$ be isomorphic skew braces.
Let  $\gamma$ and $\gamma'$ be the gamma functions attached to $M$ and $N$, respectively. Then, for each skew brace isomorphism $f\colon M\to N$, and for each $x\in N$, the following holds
\begin{equation}
\label{eq:iso}
\gamma'_x=f\gamma_{f^{-1}(x)}f^{-1}.
\end{equation}
\end{lemma}
\begin{proof}
Let $x, y\in N$ and $a,b\in  M$ be such that $f(a)=x$ and $f(b)=y$. Then,  
\begin{equation*}\gamma'_x(y)=-'x+'x\circ' y=-'f(a)+' f(a)\circ' f(b)=f(-a+a\circ b) =f\gamma_a(b).
\vspace{-\baselineskip}
\end{equation*}
\end{proof}
The following corollary is an obvious consequence of the analogue result for groups and the previous lemma.
\begin{corollary}
An isomorphism between two skew braces $M$ and $N$  induces a one to one correspondence between the subskew braces, the left ideals and the ideals of $M$ and $N$.
\end{corollary}

\subsection{Splitting of skew braces}

The direct product of two skew braces $L$ and $M$  is the skew brace on the set $L\times M$  with additive (resp. multiplicative) group given by the direct product of the additive (resp. multiplicative) groups of $L$ and $M$. 

If $\gamma_L$ and $\gamma_M$ are the gamma function associated to $L$ and $M$ respectively, then the gamma function associated to $L\times M$ is  the map which associate to $(l,m)\in L\times M$ the automorphism $((\gamma_L)_l, (\gamma_M)_m)$.

If a skew brace $N$ is isomorphic to $L\times M$, then $L$ and $M$ are (isomorphic to) ideals of $N$.

In \cite[Theorem 1]{Byo13} is proven  that if both groups of a finite skew brace $N$ are nilpotent then $N$ is the direct product of skew braces of prime power order. The same has been also reobtained in \cite[Theorem 2.6]{CCDC20} using the gamma functions. The argument given there can be generalised to obtain the following.
\begin{prop}
\label{prop:dirsum}
{\sl
Let $N$ be a skew brace with associated gamma function $\gamma$. Assume that the additive group of $N$ is a direct sum of subgroups $N=\bigoplus\limits_{j=1}^t I_j$, and for all $n\in N$ use the notation $n=\sum_{j=1}^t n_j$ with $n_j\in I_j$.

  The following facts are equivalent.
\begin{enumerate}[label=(\roman*)]
\item The map $f\colon N\to  I_1\times\dots\times I_t$ defined by $f(\sum_{j=1}^t x_j)=(x_1,\dots, x_t)$ is a  skew brace isomorphism.
\item $I_1, \dots, I_t$ are ideals of $N$.
\item 
For all $x,y\in N$
$$\gamma_x(y)=\sum_{j=1}^t\gamma_{x_j}(y_j)$$
\item 
For all $x,y\in N$
$$x\star y=\sum_{j=1}^t x_j\star y_j.$$
\end{enumerate}
}
\end{prop}
\begin{proof}
$(i)\Rightarrow (ii)$ is clear, and $(ii)\implies (i)$ is \cite[Theorem~4.2]{CSV19}.

$(i)\Rightarrow (iii)$ follows by recalling that the gamma functions associated to the ideals $I_j$ are the restrictions of $\gamma$, so, denoting by $\gamma'$ the  gamma function  associated to their direct product, we have
$$\gamma'_{(x_1,\dots,x_t)}=(\gamma_{x_1},\dots,\gamma_{x_t}).$$
Therefore, using \eqref{eq:iso} we compute
$$\gamma_x(y)=f^{-1}\gamma'_{f(x)}(f(y))=
f^{-1}\gamma'_{(x_1,\dots,x_t)}(y_1,\dots,y_t)=f^{-1}(\gamma_{x_1}(y_1),\dots,\gamma_{x_t}(y_t))=
\sum_{j=1}^t\gamma_{x_j}(y_j).$$
As for the converse, we have to prove that the map $f$, which is known to be an isomorphism of additive groups, is also a morphism of multiplicative groups, namely
$$f(x\circ y)=f(x)\circ' f(y)$$
where $\circ'$ denotes the product in $I_1\times\dots\times I_t$.
Now, 
\begin{equation*}
\begin{aligned}
f(x\circ y)=&f(x+\gamma_x(y))
=f(\sum_{j=1}^tx_j+\sum_{j=1}^t\gamma_{x_j}(y_j))\\
=&(x_1+\gamma_{x_1}(y_1),\dots,x_t+\gamma_{x_t}(y_t))
=(x_1\circ y_1,\dots, x_t\circ y_t)
=f(x)\circ'f(y).
\end{aligned}
\end{equation*}
Finally, the equivalence from $(iii)$ and $(iv)$ easily follows from \eqref{eq:*gamma}.
\end{proof}

\subsection{Nilpotency and solubility of a skew brace}

The operation $\star$ associated to a skew brace $(N, +,\circ)$ is a commutator in ${\rm Hol}(N)=\Aut(N)\ltimes N$, namely
\begin{equation}
\label{eq:op-star}
x\star y= \gamma_x(y)-y=[\gamma_x,y],
\end{equation}
 and it is a measure of the difference between the two operations.

For $X,Y$ non empty subsets of a  skew brace $N$, we let $X\star Y$ be the additive subgroup of $(N, +)$ generated by the elements $x\star y$ for $x\in X$ and $y\in Y$. Again we can describe this set in terms of commutators in ${\rm Hol}(N)$, 
namely
$$X\star Y=[\gamma(X), Y].$$

We define $N^{(1)}=N^{1}=N_{(1)}=N$ and inductively
\begin{equation}
N^{(k)}=N^{(k-1)}\star N,\quad
N^{k}=N\star N^{k-1},\quad
N_{(k)}=N_{(k-1)}\star N_{(k-1)}
\end{equation}
 for all $k\ge2$. Also for $k=2$ we have the  equalities $N^{(2)}=N^2=N_{(2)}$.

\subsubsection{The right series} All the elements of the series  $N^{(k)}$ are ideals of the skew brace (see \cite[Proposition~2.1]{CSV19}), and therefore the chain  

\begin{equation}
\label{eq:filtration}
N=N^{(1)}\supseteq N^{(2)}\supseteq N^{(3)}\supseteq \dots,
\end{equation}
 is a filtration for the skew brace $N$, and is called the \emph{right series} of the skew brace.

The filtration of the right series is particularly interesting since it is has the property that the skew brace structure induced on  the quotients of the filtration is the trivial one. 
So, this  series makes it clear that there  are important links between  the additive and the multiplicative structure of  a skew  brace.

%

A skew brace $N$ is called \emph{right nilpotent} of class $m$ if $N^{(m)}\supsetneq N^{(m+1)}=0$ for some $m\ge1$.

The following lemma is  \cite[Proposition 2.3]{CSV19}.
 \begin{lemma}
 \label{lemma:csv2.3}
 {\sl
$N^{(2)}$ is the smallest ideal of $N$ such that $N/N^{(2)}$ is a trivial skew brace.
}
\end{lemma}
\begin{remark}
\label{rem:s-esatta}
{\rm
In the case when $N^{(2)}$ is a trivial brace, Lemma~\ref{lemma:csv2.3}  says that
 both $(N, +)$ and $(N,\circ)$ are extensions of $N/N^{(2)}$ by $N^{(2)}$.
}
\end{remark}

\subsubsection{The left series}
All the elements of the series  $N^{k}$ are left ideals of the skew brace (see \cite[Proposition~2.2]{CSV19}), and therefore the chain  

\begin{equation}
\label{eq:filtration}
N=N^{1}\supseteq N^{2}\supseteq N^{3}\supseteq \dots,
\end{equation}
 is a filtration for both the additive and the multiplicative group of the skew brace $N$. This series is called the \emph{left series} of the skew brace.
 As we noticed before, $x\star y$ is the commutator in ${\rm Hol}(N)$ of $\gamma_x$ and $y$, therefore 
 \begin{equation*}
  N^k=   [    \underbrace{\gamma(N),     \dots,
          \gamma(N)}_{k-1}, N]
  \end{equation*}
is an iterated commutator.

A skew brace $N$ is called \emph{left nilpotent} of class $m$ if $N^{m}\supsetneq N^{m+1}=0$ for some $m\ge1$.

The characterisation of $N^k$ as an iterated commutator immediately gives the following (see also \cite[proposition~4.4]{CSV19})
\begin{prop}
\label{prop:lnilp}
{\sl
Let $p$ be a prime and $N$ be a skew brace of cardinality $p^m$. Then $N$ is left nilpotent of class $\le m$.
}
\end{prop} 

\section{Module braces}

 \label{sec:modulebraces}
 In this section we restrict our study to braces, namely to skew braces with an abelian additive group, whose additive groups have an additional structure of modules over a ring $R$. 
 \begin{definition}
 \label{def:modulebrace}
  Let $(N,+,\circ)$ be a brace with gamma function $\gamma\colon N\to {\rm Aut}(N)$.
Assume that $(N,+)$ has  a structure of  left (right) module over some ring $R$. 

We say that $(N,+,\circ)$ is a (left/right) \emph{$R$-module brace} (an $R$-brace for short) if  $\gamma(a)$ is a  $R$-module automorphism of $N$, for each $a\in N$, namely if $\gamma(N)\subseteq{\rm Aut}_R(N)$.
\end{definition}
This notion clearly generalises the classic notion of braces, which we are now calling $\Z$-braces.
This also extend the notion of $F$-brace, where $F$ is a field, given in \cite[Definition~2]{CCS19}. In fact, following the  definition given there, a brace $N$ is an  $F$-brace if
$$-rx+r(x\circ y)=-x + x\circ(r y),$$
for all $x,y\in N$ and $r\in F$,
which in terms of the gamma function is  
$$r\gamma_x(y)=\gamma_x(ry),$$
namely it is equivalent to $\gamma\in\Aut_R(N)$.
The same condition, in terms of the $\star$ operation, becomes $$r(x\star y)=x\star ry.$$
(See also \cite{Smo22a, Smo22b}).

We recall that when $R$ is commutative each left $R$-modulo is also a right $R$-modulo and viceversa.
\begin{example}
{\rm
An $R$-module $N$ with the trivial brace structure is always an $R$-brace, since in this case the corresponding gamma function is the trivial map $x\mapsto\gamma_x={\rm id}$ for all $x\in N$.
}
\end{example}
\begin{example}
\label{ex:radringRbrace}
{\rm
Let $N=(N,+ ,\cdot)$ be a radical ring and let $\circ$ the adjoint operation on $N$. 
If $(N, +)$ has a right $R$-module structure, we have that $(N,+,\circ)$ is an $R$-module brace. In fact the gamma function associated to this brace is given by $\gamma_x(y)=-x+x+y+xy=(1+x)y$, for all $x, y\in N$, so $\gamma_x\in{\rm Aut}_R(N)$ for all $x$.
}
\end{example}
\begin{example}
\label{ex:RbraceNoradring1}
{\rm
Let $R=\Z[i]$ be the ring of Gaussian integers, and let $N$ be the additive group $\Z[i]\times\Z[i]$. Consider on $N$ the operation $\circ$ defined by
\begin{equation}
\label{eq:op-es}
(\alpha_1, \beta_1)\circ (\alpha_2, \beta_2)=(\alpha_1+(-1)^{\Re(\alpha_1)}\alpha_2, \beta_1+(-1)^{\Re(\alpha_1)}\beta_2)
\end{equation}
for all $(\alpha_1, \beta_1),(\alpha_2, \beta_2)\in \Z[i]\times\Z[i]$ (here $\Re(\alpha_1)$ denotes the real part of $\alpha_1$).
We claim that $(N,+,\circ)$ is an $R$-brace. An easy way to prove this is to show that the map
$$\gamma\colon N\to \Aut_R(N)$$
defined by 
$(\alpha, \beta) \mapsto\gamma_{(\alpha, \beta)}=(-1)^{\Re(\alpha)}id$ is well defined and verifies the gamma functional equation \eqref{eq:GFE}. Clearly, $\gamma_{(\alpha, \beta)}=\pm id\in\Aut_R(N)$; moreover
\begin{equation*}
\begin{aligned}
\gamma((\alpha_1, \beta_1)+\gamma_{(\alpha_1, \beta_1)}(\alpha_2, \beta_2))=&
\gamma((\alpha_1, \beta_1)+(-1)^{\Re(\alpha_1)}(\alpha_2, \beta_2))\\
=&
\gamma((\alpha_1+(-1)^{\Re(\alpha_1)}\alpha_2, \beta_1+(-1)^{\Re(\alpha_1)} \beta_2))\\
=&(-1)^{\Re(\alpha_1)+(-1)^{\Re(\alpha_1)}\Re(\alpha_2)}\\
=&(-1)^{\Re(\alpha_1)}(-1)^{\Re(\alpha_2)}\\
=&\gamma((\alpha_1, \beta_1))\gamma((\alpha_2, \beta_2)).
\end{aligned}
\end{equation*}
Finally, it is immediate to see that the $\circ$ operation associated to $\gamma$ is \eqref{eq:op-es}.

This example has been constructed by using the general result of \cite[Proposition~2.14]{CCDC20}.
}
\end{example}

\begin{lemma}
\label{lemma:dirprod}
{\sl
Let $M$ and $N$ be $R$-braces;  then their direct product $M\times N$, with the $R$-module structure given by the direct product of $R$-modules, is an $R$-brace. 
}
\end{lemma}
\begin{proof}
The lemma follows by noticing that  if $\gamma$ and $\gamma'$ are the gamma functions associated to $M$ and $N$ respectively, then the $\gamma$-function  associated to $M\times N$ is the map 
$$\gamma\times\gamma'\colon M\times N\to \Aut_R(M)\times\Aut_R(N)\subseteq\Aut_R(M\times N).$$
\vspace{-\baselineskip}
\end{proof}

\smallskip
 
The following example shows that not all braces $N$ for which the additive group has also an $R$-module structure are $R$-braces. 
\begin{example}
{\rm
Let $N=\Z[i]$ considered as a $\Z[i]$-module, and let  $\gamma\colon\Z[i]\to\Aut(\Z[i])$ be the map defined by 
$$\gamma_{(a+ib)}(x,y)=((-1)^a x+{ iy}).$$
It is easy to verify that $\gamma$ is  a gamma function, so $(N,+,\circ)$ is a brace. However, $N$ is not  a $\Z[i]$-brace, since $\gamma_{(a+ib)}\not\in\Aut_{\Z[i]}(\Z[i])$ for $a$ odd.
}
\end{example}
%
%
%
%
%
%
%
%
\color{black}
The next easy lemma (which will be useful in the following) shows that the restriction of scalars preserves the module brace structure.
\begin{lemma}
\label{lemma:lifting}
{\sl
Let $S$ and $T$ be  rings and let $f\colon T\to S$ be a ring homomorphism. Then each $S$-brace $M$ has a natural  $T$-brace structure via $f$.
}
\end{lemma}
\begin{proof}
It is well known that $M$ is a $T$-module via $f$. On the other hand, if $\varphi\in {\rm Aut}_S(M)$ then 
$t\cdot\varphi(m)=f(t)\varphi(m)=\varphi(f(t)m)=\varphi(t\cdot m)$ for all $t\in T$ and $m\in M$, so 
$\varphi\in {\rm Aut}_T(M)$ and $M$ is a $T$-brace.
\end{proof}

In the notation of the previous lemma, we say that $M$ is a $T$-brace \emph{by restriction of scalars} (via $f$). 
\smallskip

Most of the definitions one can give for braces have an analog for $R$-braces.
Two $R$-braces $N$ and $M$ are \emph{isomorphic} if there exists a braces isomorphism $\varphi\colon N\to M$ such that is also an $R$-module isomorphism. 

 
   An   \emph{ $R$-subbrace} of an $R$-brace $N$  is a subbrace which is also an $R$-submodule,
  This definition is justified by the following lemma.
 \begin{lemma}
 \label{lemma:Rsubb} 
{\sl
Let $N$ be an $R$-brace and let $M$ be a subbrace and an $R$-submodule of  $N$. Then $M$ with the induced operations is  an $R$-brace.
}
  \end{lemma} 
\begin{proof}
A subbrace is a brace with the induced operations. 
So, if $\gamma\colon N\to\Aut_R(N)$ is the  gamma function associated to the $R$-brace $N$, then the gamma function associated to the brace $M$ is the restriction of $\gamma$ to $M$, more precisely the map 
$$\gamma_M\colon M\to \Aut(M)$$
sending $m\mapsto(\gamma_m)_{|M}$:  (this map is well defined since $M$ is a subbrace of $N$, so $\gamma_m(m')=-m+m\circ m'\in M$ for all $m,m'\in M$).
On the other hand,  since  $\gamma_m$ is an $R$-module automorphism for all $m$, then also its restriction to the $R$-submodule  $M$  is such,  so 
$(\gamma_m)_{|M}\in \Aut_R(M)$. 
\end{proof}

   A left ideal $I$ of an $R$-brace $N$ is called  \emph{left $R$-ideal} if $I$ is also an $R$-submodule of $N$. Equivalently, a left $R$-ideal of $N$ is an $R$-subbrace of $N$ which is $\gamma(N)$-invariant  (in \cite{CCS15} an analogous definition is given for the case when $R$ is a  field). 
In the same way, an  \emph{$R$-ideal} of an $R$-brace $N$ is an ideal of $N$ which is also an $R$-submodule of $N$.

\begin{lemma}
\label{lemma:quotient}
{\sl 
Let  $(N, +,\circ)$ be an $R$-brace, and let $I$ be an $R$-ideal of $N$. Then the quotient brace $N/I$ is an $R$-brace.
}
\end{lemma}
\begin{proof}
We already know that $N/I$ is a brace and an $R$-module. Now, if $\gamma$ is the gamma function associated to $N$, then the gamma function $\bar\gamma$ associated to the quotient structure on $N/I$ is 
the map defined by 
$\bar\gamma_{x+I}=\pi\gamma_x$, where $\pi\colon N\to N/I$ denotes the projection, so it is an $R$-module isomorphism.
\end{proof}

%
%
%
%

\begin{lemma}
\label{lemma:*op}
{\sl
Let $(N,+,\circ)$ be an $R$-brace. Then for each $k\ge1$ the following hold.
\begin{enumerate}[label=(\roman*)]
 \item $N^{(k)}$ is an $R$-ideal of $N;$
 \item $N^k$ is a left $R$-ideal of $N$;
 \item $N_{(k)}$ is an $R$-ideal of $N_{(k-1)}$ and an $R$-subbrace of $N$.
 \end{enumerate}
 }
\end{lemma} 
 \begin{proof}
(i)  By  \cite[Proposition~2.1]{CSV19},  $N^{(k)}$is  an ideal of the brace $N$;  moreover  $N^{(k)}$ is closed under multiplication by elements of $R$ since, for $x\in N^{(k-1)}$, $y\in N$ and $r\in R$,
$$r(x\star y)=r(\gamma_x(y)-y)=r\gamma_x(y)-rx=\gamma_x(ry)-ry=x\star ry\in N^{(k)}$$
being $\gamma_x$ an $R$-module morphism. 

(ii)  By  \cite[Proposition~2.1]{
CSV19}  $N^{k}$ a left ideal of the brace $N$ and the equality  
$r(x\star y)=x\star ry$ together with an inductive argument shows that 
$N^k$ is an $R$-module.

(iii)  $N_{(k)}=N_{(k-1)}\star N_{(k-1)}$, so it is an $R$-ideal of $N_{(k-1)}$ by part (i), and therefore it is an $R$-subbrace of $N$. 
 \end{proof}
%
%
%
%
%
%
%
%

 \color{black}
  \subsection{Splitting of module braces}

  The following proposition can be seen as a generalisation of \cite[Theorem~1]{Byo13}. 
  \begin{prop}
\label{prop:decomposition}
{\sl
Let $R$ be a commutative ring, with unit $1$, which is the direct sum of its subrings $R_i$ (with unit $e_i$)
$$ R=\bigoplus_{i=1}^t R_i.$$
Let $(N,+)$ be an abelian group which is also an $R$-module. Then $N$ can be written in the form
\begin{equation}
\label{eq:N=}
N=\bigoplus_{i=1}^t N_i,
\end{equation}
where $N_i=e_iN$ is an $R$-module, which is annihilated by $R_j$ for all $j\ne i$.

If, in addition,  $(N,+)$ has also a  structure $(N,+,\circ)$ of  $R$-brace, 
then, for each $i$, $N_i$ is a left $R$-ideal of $N$. 

Moreover, all the $N_i$'s are normal in $(N, \circ)$, or, equivalently, are ideals of $N$, if and only if
 the decomposition in \eqref{eq:N=} is an $R$-braces decomposition.


}
 
\end{prop}
\begin{proof}
%

The elements $e_1,\dots, e_t$ are orthogonal idempotents of $R$ and $e_1+\dots +e_t=1$,
so any $x\in N$ can be written as $x=1x=e_1x+\dots+ e_tx$.  Further, the decomposition of $x$ into a sum of elements of the $N_i$'s is unique since 
if $x=e_1y_1+\dots+ e_ty_t$, by multiplying by $e_i$ we get $e_ix=e_i( e_1y_1+\dots+ e_ty_t)=e_iy_i$. This proves the equality in  \eqref{eq:N=}.

Moreover, for all $r_j\in R_j$ and  $x\in N$ we have $r_j e_i x= r_je_je_i x=0$, so $N_i$ is  annihilated by $R_j$ for all $j\ne i$.

We also notice that  
 $N_i=e_iN$  is invariant under the action of $\Aut_R(N)$, for all $i$.  Therefore, if  $(N,+)$ has also a  structure $(N,+,\circ)$ of  $R$-brace, then all $N_i$'s are left ideals of $N$, and the last part of the statement follows from Proposition~\ref{prop:dirsum}.
\end{proof}
In view of Example~\ref{ex:radringRbrace}, as a particular case of the previous proposition we get the following.

\begin{corollary}
\label{cor:brace-dec}
{\sl
Let $R=\bigoplus_{i=1}^tR_i$ be a commutative ring with identity.
Let $(N,+,\cdot)$ be a radical ring, and let $(N,+,\circ)$ be the associated brace.
If  $(N,+)$ has also an $R$- module structure, then, 
with the notation of Proposition~\ref{prop:decomposition},
 we have 
\begin{equation}
\label{eq:N=radring}
N=\bigoplus_{i=1}^t N_i
\end{equation}
 as $R$-braces, namely 
$$1+N= \bigoplus_{i=1}^t(1+Ne_i).$$
}
\end{corollary}
\begin{proof}
In Example~\ref{ex:radringRbrace} we showed that in this case  the brace associated to the radical ring $N$ is an $R$-brace, so,
by Proposition~\ref{prop:decomposition},   the equality in \eqref{eq:N=radring} is a brace decomposition if and only if the left ideals $N_1, \dots, N_t$ are ideals of $N$. 

In view of Proposition~\ref{prop:dirsum}, this is equivalent to showing that 
$$x \cdot y=\sum_{i=1}^te_ix\cdot e_iy,$$  
for all $x,y\in N$ (recall that the $\star$ operation of a radical ring is just the product of the ring), and this easily follows from the orthogonality of the $e_i$'s.
\end{proof}

%

%

\begin{remark}
\label{rem:decomposition}
{\rm
If $R=\Z$ and $(N, +,\circ)$ is a finite brace, then the action of $\Z$ on $N$ can not be faithful, so  $N$ is a $\Z/d\Z$-module for some $d$.

Let  $d=p_1^{a_1}\dots p_t^{a_t}$, where the $p_i$'s are pairwise distinct primes, then 
$$\Z/d\Z\cong\bigoplus_{i=1}^t\Z/p_i^{a_i}\Z$$
and, with the notation of Proposition~\ref{prop:decomposition},  $N_i$ is the Sylow $p_i$-subgroup of $(N,+)$, and it is also a Sylow $p_i$-subgroup of $(N,\circ)$. Therefore, if $(N,\circ)$ is nilpotent we recover the core of \cite[Theorem~1]{Byo13} (see also \cite[Corollary~4.3]{CSV19} and \cite[Subs. 2.4]{CCDC20}).
}
\end{remark}

 \section{R-braces of small rank}
\label{sec:smallrank}

In this section we present our result on the relations between the additive and the multiplicative groups of a $R$-brace. In  \cite[Theorem~1]{FCC} and its generalisation  in \cite[Theorem~2.5]{Bac16}, these relations are studied for ($\Z$-)braces. Building upon these results, in Theorem~\ref{teo:isomorphic}, we prove that  in the case $D$-braces, where $D$ is a  principal ideal domain with some additional assumptions, the connections between 
the additive and the multiplicative groups can be much tighter, depending on D.

(In a different direction  \cite[Theorem 2.5]{Bac16} has been recently generalised in \cite{CDC22}.)

%

Theorem 4.1 gives results for module braces on rings of a particular type. In the remainder of the section, after introducing some basic on $p$-adic fields, we show that this theorem can be applied to module braces on certain $p$-adic rings; then, in Lemma~\ref{lemma:rango} and Corollary~\ref{cor:rank}, we quantify the advantage that in this case one can obtain by applying Theorem 4.1 instead of \cite[Theorem~2.5]{Bac16}. Next, we show, with an argument traceable to Hensel's Lemma, how Theorem 4.1 can be applied to modules on local rings (see Lemma~\ref{lemma:padic}). Finally, we show how to apply the results obtained in these cases to get information on $R$-braces, with R any commutative ring.


\smallskip

Let $D$ be a principal ideal domain (PID), and let $M$ be a  torsion $D$-module. $M$ has a unique decomposition as a sum of 
\emph{indecomposable} cyclic $D$-modules. We call  \emph{$D$-rank} of $M$   the number of cyclic factors of this decomposition and denote it as ${\rm rank}_DM$. 

 \begin{theo}
\label{teo:isomorphic} 
{\sl
Let $p$ be a prime number, and let $D$ be a PID such that $p$ is a prime in $D$.
Let $(N,+,\circ)$ be a $D$-brace of order a power of $p$.

Assume that $r=\rm{rank}_D N<p-1$. 

Then $(N,+)$ and $(N,\circ)$ have the same number of element of each order.
In particular, if $(N,\circ)$ is abelian, then $(N,+)\cong (N,\circ)$.
}
\end{theo}

Sometimes we will say generically that a $D$-brace $N$ as in the theorem has \emph{``small''  $D$-rank}.

\begin{remark}
{\rm \cite[Theorem 2.5]{Bac16}  gives the same result of Theorem~\ref{teo:isomorphic} for  $D=\Z$. Further examples of rings $D$ as in the theorem are given by the ring of integers of unramified extensions of the field of the $p$-adic numbers $\Q_p$.
}
\end{remark}

Theorem~\ref{teo:isomorphic} will be proved in Section~\ref{sec:proof}. 
The next corollary makes it explicit for radical rings with a $D$-algebra structure.

\begin{corollary}
\label{cor:radring}
 Let $(N,+,\cdot)$ be a nilpotent commutative ring of order a power of a prime $p$. Let $D$ be a PID in which $p$ is prime. If  $N$ is a $D$-module of ${\rm rank}_DN<p-1$, then $1+N\cong N.$
\end{corollary}
\begin{proof} The brace associated to $N$ is a $D$-brace (see Example~\ref{ex:radringRbrace}), therefore, since 
for a nilpotent ring  $(1+N,\cdot)\cong(N,\circ)$, Theorem~\ref{teo:isomorphic} applies.
\end{proof}

In the following we show that the assumption on the ring $D$ in  Theorem~\ref{teo:isomorphic} is not so restrictive, since in many cases it can be satisfied with  $D\ne\Z$.

\subsection{Module braces over $p$-adic rings}
As we noted above, the rings of integers of unramified extensions of the field $\Q_p$ enjoy the properties prescribed for the ring $D$ in Theorem~\ref{teo:isomorphic}. We now recall the definition and some of their properties, which can be found for example in \cite[Ch. II, \S4]{Lang94ANT}, \cite[Ch. III, \S 5]{Ser95}.

Let $p$ be a prime. A finite extension $K/\Q_p$ is called \emph{unramified} if the ideal generated by $p$ is prime in the ring of integer $\mathcal O_K$ of $K$. Since  $\mathcal O_K$ is a discrete valuation ring, the only ideals of $\mathcal O_K$ are the powers of  the maximal ideal, that in this case is $p\mathcal O_K$.
\nota{Rivedere}
 \begin{prop}
\label{prop:unram}
{\sl
Let $\bar \Q_p$ be a fixed algebraic closure of $\Q_p$. For each $\lambda\ge1$ there exists a unique unramified extension of $\Q_p$  of degree $\lambda$ contained in $\bar\Q_p$.
This extension is  a Galois extension of $\Q_p$, and it is generated by any root of any monic polynomial $\tilde g(x)\in\Z_p[x]$ of degree $\lambda$, whose reduction modulo $p$ is irreducible.  
The ring of integers  of such extension is  a discrete valuation ring with maximal ideal $(p)$, and it is generated over $\Z_p$ by any root of $\tilde g(x)$.
}
\end{prop}

For each prime $p$ and for each positive integer $\lambda$ we will denote by $\Q_p(\lambda)$ the unique unramified extension of $\Q_p$  of degree $\lambda$ contained in $\bar\Q_p$, and by
 $\Z_p(\lambda)$ its ring of integers.

The following lemma and the next corollary quantify the gain we have by appealing  to  Theorem~\ref{teo:isomorphic} instead of \cite[Theorem~1]{Bac16} when dealing with a $\Z_p(\lambda)$-brace. 
\begin{lemma}
\label{lemma:rango}
{\sl
Let $M$ be a finite $\Z_p(\lambda)$-module. Then,
$${\rm rank}_\Z(M)=\lambda\cdot {\rm rank}_{\Z_p(\lambda)}(M).$$
}
\end{lemma}

\begin{proof}
Consider the decomposition of $M$ as a sum of cyclic $\Z_p(\lambda)$-modules
$$M\cong\bigoplus_{i=1}^r \frac{\Z_p(\lambda)}{p^{c_i}\Z_p(\lambda)}.$$
Recalling that 
$\Z_p(\lambda)$ is  a free $\Z$-module of rank $\lambda$, we get
$$M\cong\bigoplus_{i=1}^r \frac{\Z_p(\lambda)}{p^{c_i}\Z_p(\lambda)}\cong\bigoplus_{i=1}^r\frac{\Z^\lambda}{p^{c_i}\Z^\lambda}\cong
\bigoplus_{i=1}^r\left(\frac{\Z}{p^{c_i}\Z}\right)^\lambda,$$
namely, if ${\rm rank}_{\Z_p(\lambda)}(M)=r$, then ${\rm rank}_\Z(M)=\lambda\cdot r$.
\end{proof}

Theorem~\ref{teo:isomorphic} and the previous lemma immediately give the following.
\begin{corollary}
\label{cor:rank}
{\sl
Let $(N,+,\circ)$ be a $\Z_p(\lambda)$-brace of order a power of $p$,  and assume that $\rm{rank}_\Z N<\lambda(p-1)$. 
Then $(N,+)$ and $(N,\circ)$ have the same number of element of each order.
In particular, if $(N,\circ)$ is abelian, then $(N,+)\cong (N,\circ)$.
}
\end{corollary}
%
%
%
\begin{remark}
\label{rem:nrings}
{\rm Let $K$ be a finite extension of $\Q_p$. Denote by  $\mathcal O_K$  its ring of integers and let $P_K$ be its maximal ideal.
If $[\mathcal O_K/P_K:\F_p]=\lambda$, then $\Z_p(\lambda)\subseteq\mathcal O_K$ (this is a classical fact and an easy consequence of  \cite[Proposition10]{Lang94ANT}), so every $\mathcal O_K$-brace is also a 
$\Z_p(\lambda)$-brace by restriction of the action (see Lemma~\ref{lemma:lifting}), therefore we can apply the previous corollary to those having a ``small" $\Z_p(\lambda)$-rank.

In the next subsection we will discuss a generalisation of this argument to module braces over a general local ring.
}
\end{remark}

\subsection{Module braces over local rings}
In the following, for a local ring $S$ we will use the notation $(S, \mathfrak m, \F_{p^{\lambda}})$ to encode the information on the ring $S$, namely to intend that  $\mathfrak m$ is its maximal ideal  and that its residue field $S/\mathfrak m$ is isomorphic to  $ \F_{p^{\lambda}}$. In this case, if $S$ is finite its characteristic is a power of $p$.

\begin{lemma}
{\sl
\label{lemma:padic}
Let $(S, \mathfrak m, \F_{p^{\lambda}})$ be a finite  local ring.
Then every $S$-brace is also a $\Z_p(\lambda)$-braces, by restriction of scalars.
}
\end{lemma}
\begin{proof}
By Lemma~\ref{lemma:lifting}, we are reduced to show that 
there exists a ring homomorphism  $\Z_p(\lambda)\to S$.

Let ${\rm char}(S)=p^c$.
The residue field $S/\mathfrak m$ is 
 of the form 
$\F_p(\bar\xi)$, for some $\xi\in S^*$ , with $\bar\xi$ of degree $\lambda$ over $\F_p$ (here $\bar\xi$ denotes the class of $\xi$ modulo $\mathfrak{m}$). 

Let $\bar g(x)\in\F_p[x]$ be the minimal polynomial of $\bar\xi$ over $\F_p$, and  let $g(x)\in \frac{\Z}{p^c\Z}[x]$ be a monic lifting of $\bar g(x)$. 

Up to changing $\xi$ with an other element in the same class modulo $\mathfrak m$, we can assume that $g(\xi)=0$.
This can be easily shown by arguing as in the classical proof of the Hensel's Lemma (see for example \cite[Section 2.2]{Ser95}). In fact, $\bar g(\bar\xi)=0$ means that $g(\xi)\in \mathfrak m$. On the other hand,  $g'(\xi) \not \in\mathfrak m$ since $\bar g(x)\in \F_p[x]$ is  irreducible and $\F_p$ is a perfect field, so its  derivative  $\bar g'(x)$ is not zero, hence it  does not vanish in $\bar\xi$. Therefore, $g'(\xi)$ is invertible in $S$, and, denoting by $u$ its inverse, we have that $\xi_1=\xi-g(\xi)u\in S$ and is in the same class of $\xi$ modulo $\mathfrak m$. Now,
$$g(\xi_1)=g(\xi-g(\xi)u)\equiv g(\xi)-g(\xi)ug'(\xi)\equiv0\pmod{g(\xi)^2},$$
so $g(\xi_1)\in \mathfrak m^2$.
Since $\mathfrak m$ is nilpotent, by iterating this construction,  in a finite number of steps we get what claimed.

Consider the composition map 
$$\psi\colon\Z_p[x]\to \frac{\Z}{p^c\Z}[x]\to S$$
where the first map is the reduction modulo $p^c$ and the second is the evaluation homomorphism $x\mapsto\xi.$
Let $\tilde g(x)\Z_p[x]$ be any monic lifting of $g(x)$.
Clearly, $\psi(\tilde g(x))=g(\xi)=0$, so $\psi$ induces a homomorphism 
$\psi'\colon\Z_p[x]/(\tilde g(x))\to S$.
Finally, by Proposition~\ref{prop:unram}, $\Z_p[x]/(\tilde g(x))\cong \Z_p(\lambda)$ and this concludes the proof.
\end{proof}

Corollary~\ref{cor:rank} and the previous lemma immediately give the following.
\begin{corollary}
\label{cor:rank-local}
{\sl
Let $(S, \mathfrak m, \F_{p^{\lambda}})$ be a  local ring, and let $(N,+,\circ)$ be a $S$-brace of order a power of $p$.
If $\rm{rank}_\Z N<\lambda(p-1)$,
then $(N,+)$ and $(N,\circ)$ have the same number of element of each order.
In particular, if $(N,\circ)$ is abelian, then $(N,+)\cong (N,\circ)$.
}
\end{corollary}
\begin{proof}
First of all, we notice that, being $N$ finite, up to changing $S$ with its quotient by the annihilator of $N$, we can assume that $N$ is a faithful module and therefore also $S$ is finite
to be finite. Clearly this does not change its property of $S$ being local or its residue field. Therefore, by Lemma~\ref{lemma:padic}, $N$ has naturally a $\Z_p(\lambda)$-structure and we can apply Corollary~\ref{cor:rank}.
\end{proof}

\subsection{Finite module brace over a general ring}

We now show  that the results obtained for module braces over a local ring can be  applied to finite any $R$-brace $N$, where $R$ is any commutative ring.

Let ${\rm Ann}_R(N)$  be the annihilator of $N$ in $R$, and put $A=R/{\rm Ann}_R(N)$.
Then $N$ is a faithful $A$-module with the induced structure, and being $N$ finite, the ring $A$ is finite, and therefore artinian. By the structure theorem of artinian rings, we have  
\begin{equation}
\label{eq:decA}
A=\bigoplus_{i=1}^t A_i
\end{equation}
where the $A_i$'s are local artinian (in this case finite) rings, of characteristic a power of a prime $p_i$, say. We fix the notation $(A_i,\mathfrak m_i, \F_{p_i^{\lambda_i}})$.

%
Let $e_1,\dots, e_t$ be the orthogonal idempotents of $A$ such that $e_1+\dots +e_t=1$ and $A_i=Ae_i$.
By Proposition~\ref{prop:decomposition},
letting  $N_i=Ne_i$, we have 
\begin{equation}
\label{eq:decN} 
N=\bigoplus_{i=1}^t N_i,
\end{equation}
where this equality holds for $(N,+)$, and for $N$ as a brace in the case when $N_1,\dots, N_t$ are ideals of $N$.
In any case, since  $N_1,\dots, N_t$ are always left ideals, some information on the $A$-brace ($R$-brace) $N$ can be derived from the study of its subbraces $N_1,\dots, N_t$, which are braces over a finite local ring. 

We can then apply the results obtained above to this general case, getting 
the following.
\begin{prop}
\label{prop:padic2}
{\sl 
Let $R$ be a commutative  ring and  
let $(N,+,\circ)$ be a finite $R$-brace.

Let $A=R/{\rm Ann}_R(N)$, 
and consider the decomposition of $A$ into a sum  of finite local rings $(A_i,\mathfrak m_i, \F_{p_i^{\lambda_i}})$ as in \eqref{eq:decA}, and the decomposition of $N$ given in \eqref{eq:decN}.
The following hold.
\begin{enumerate}
\item 
For each $i=1,\dots, t,$ the subbrace $N_i$ has cardinality a power of $p_i$ and  is a $\Z_{p_i}(\lambda_i)$-brace.
\item
If $N$ has order a power of a  prime $p$, then $p_i=p$ for all $i$, and $N$ is a $\Z_p(\lambda)$-brace where $\lambda=\gcd\{\lambda_i\mid\ i=1,\dots, t\}$.
\end{enumerate}
}
\end{prop}
\begin{proof}
By Lemma~\ref{lemma:padic}  the $A_i$-brace $N_i$  is a $\Z_{p_i}(\lambda_i)$-brace. Therefore, $N_i$ is a direct sum of  cyclic and finite $\Z_{p_i}(\lambda_i)$-modules, so it is a $p_i$-group, proving (1). 

As for (2), we notice that, since $N$ is a $p$-group and a faithful $A$-module, then the characteristic of $A$ is a power of $p$. Therefore, the same is true for $A_1,\dots, A_t$, and part (1) ensures that  $N_i$ is a $\Z_p(\lambda_i)$-brace. Now, since $\lambda\mid \lambda_i$ for all $i$, then $\Z_p(\lambda)\subseteq \Z_p(\lambda_i)$, and  $N_i$ is a $\Z_p(\lambda)$-brace by restriction of the action, and so, by Lemma~\ref{lemma:dirprod} also
 $N=\bigoplus_{i=1}^tN_i$ is a $\Z_p(\lambda)$-brace.
\end{proof}
\begin{corollary}
{\sl
In the notation of Proposition~\ref{prop:padic2}, assume that $N$ has order the power of a prime $p$ and that the $\Z$-rank of  $(N,+)$ is $<\lambda(p-1).$ 
Then  $(N,+)$ and $(N,\circ)$ have the same number of element of each order, and if $(N,\circ)$ is abelian, then $(N,+)\cong (N,\circ)$.
}
\end{corollary}
\begin{proof}
By Proposition~\ref{prop:padic2}, $N$ is a $\Z_p(\lambda)$-module, so we can apply Corollary~\ref{cor:rank}.
\end{proof}

\begin{corollary}
\label{cor:finale}
{\sl 
In the notation of Proposition~\ref{prop:padic2}, assume that for some $i\in\{1,\dots,t\}$
\begin{equation}
\label{eq:rank}
{\rm rank}_\Z(N_i)<\lambda_i(p_i-1).
\end{equation}

Then, $(N_i,+)$ and $(N_i,\circ)$ have the same number of element of each order, and  if $(N_i,\circ)$ is abelian, then $(N_i,+)\cong (N_i,\circ)$.

If, in addition, $N_1,\dots, N_t$ are ideals of the brace $N$, and \eqref{eq:rank} holds for all $i$, then $(N,+)$ and $(N,\circ)$ have the same number of element of each order, and if $(N,\circ)$ is abelian, then $(N,+)\cong (N,\circ)$.
}
\end{corollary}
\begin{proof}
By Proposition~\ref{prop:padic2}, $N_i$ has cardinality a power of $p_i$, so the first part follows from Corollary~\ref{cor:rank-local} .

For the last part, we recall that if $N_1,\dots, N_t$ are ideals then, by Proposition~\ref{prop:decomposition},   $N=\bigoplus_{i=1}^t N_i$ 
as a brace, and since  $(N_i,+)$ and $(N_i,\circ)$ have the same number of element of each order, for each $i$, the same is true for their direct products. Finally,
  if also  $(N,\circ)$ is abelian, then necessarily $(N,\circ)\cong (N,+)$
\end{proof}

 \begin{remark}
\label{rem:units}
{\rm
As we  saw in the previous corollaries, interesting information on an $R$-brace can also be obtained with the application of Theorem~\ref{teo:isomorphic}, or its corollaries,  to its subbraces, even in the case when it can not be applied to the whole brace. 

Consider the case of an $R$-brace $N$, with an $R$-ideal $I$: we have the following 
exact sequence of $R$-braces
$$ 0\to I\to N\to N/I\to 0.$$
The exactness of this sequence of $R$-braces means that  both the additive  group of $(N, +)$ is an  extension of $(N/I,+ )$ by $(I,+)$, and the multiplicative group $(N,\circ)$ is an  extension of $(N/I,\circ )$ by $(I,\circ)$.
If we can somehow control the differences between the additive and the multiplicative structure of $I$ and $N/I$ (for example by applying to them Theorem~\ref{teo:isomorphic}, or its corollaries) we we can deduce information on the relation between the additive and the multiplicative group of the brace $N$.

In particular, if it can be proven that $I$ and $N/I$ are trivial braces, then both the additive and the multiplicative group of $N$ are extension of $N/I$ by $I$ (see also  Lemma~\ref{lemma:csv2.3} and Remark~\ref{rem:s-esatta}).
%
%
%
 }
  \end{remark}

\begin{remark}
\label{rem:units}
{\rm
The result of this section can be fruitfully applied for example in the study of Fuchs' question on the classification of the abelian groups that can be realised as group of units of a commutative ring.  
This is an old question, that originally appears in \cite[Problem 72]{Fuchs60}, which is far from being completely solved. 

In  \cite{dcdAMPA, dcdBLMS} R.Dvornicich and the author considered the case of \emph{finite} abelian group, and in  \cite[Theorem 3.1]{dcdAMPA} and \cite[Subsection 5.3]{dcdBLMS}  showed that most of the information on the group of units of a ring $R$ is contained into the group $1+\mathfrak N$ where $\mathfrak N$ is the nilradical of $R$. Therefore, the results of this section can be applied to the  nilpotent commutative ring $\mathfrak N$ and to the $R$-brace associated to it,  whenever $\mathfrak N$ has ``small'' $R$-rank. 
In a forthcoming paper,  we will  study Fuchs' question  further giving the detail of this application, which allows us to make interesting progresses. 
}
\end{remark}

\section{Proof of Theorem~\ref{teo:isomorphic}}
\label{sec:proof}
The proof of Theorem~\ref{teo:isomorphic} is inspired from that of \cite[Theorem~1]{FCC}. We first establish the  preliminary result.

\begin{remark}
\label{rem:ff}
{\rm
 We notice that under the assumption of the theorem  $D/pD$ is a finite field. In fact, if $N$ is a non-trivial $D$-module of cardinality a power of $p$, then the annihilator $I$ of $N$ in $D$ is contained in $pD$. Now, $N$ is  finite and  is faithful over $D/I$, so that  $D/I$, and therefore also $D/pD$, must be finite.
}
\end{remark}

The following lemma refines Proposition~\ref{prop:lnilp}.
 \begin{lemma}
\label{lemma:dimension}
Let $K$ be a  finite field of characteristic $p$ and let $(M, +, \circ)$ be a $K$-module brace of dimension $r$.
Then $M^{r+1}=0$. 
\end{lemma}
\begin{proof}
$M$ is a $p$-group, so, by Proposition \ref{prop:lnilp}, $M$ is left nilpotent, namely $M^k=0$ for some $k>0$. This implies that the left series of $M$  is strictly  decreasing  to 0
\begin{equation}
\label{eq:chain}
M\supsetneq M^{2}\supsetneq   \dots \supsetneq\{0\}.
\end{equation}

Now,  by Lemma~\ref{lemma:*op},  in  \eqref{eq:chain} we have a chain of $K$-vector spaces, and, being strictly decreasing,  $\dim_KM^{(i)}\le r-i+1$, so $M^{r+1}\nobreak=\nobreak0$
\end{proof}

We  use the  notation  $m_\circ a=a\circ a\circ\dots\circ a$ ($m$ factors). We can easily obtain the following formula (see also \cite[(3.2)]{CDC22})
\begin{equation}  \begin{aligned}
\label{eq:p_pallino_a}
p_\circ a&=(\gamma_a^{p-1} + \dots + \gamma_a + 1)(a)\\
&=(p +
      \binom{p}{2} \delta_a
      + \dots +
      \binom{p}{p-1} \delta_a^{p-2})(a)
+
  \delta_a^{p-1}(a)
\end{aligned}
\end{equation}
where we are using the notation $\delta_a=\gamma_a-id$.
  
The following proposition is the analogue of \cite[Proposition 4]{FCC}.
\begin{prop}
\label{prop:4}
Let $K$ be a finite field of characteristic $p$. Let  $(M, +, \circ)$ be  a  $K$-brace of dimension $r<p-1$. Then all the non-zero elements of $(M,\circ)$ have order $p$.

In particular, if $(M,\circ)$ is abelian $(M,\circ)$ and  $(M,+)$ are isomorphic groups.
\end{prop}
\begin{proof}
Let $K=\F_{p^\lambda}$, then $K$ as a group is isomorphic to $(\Z/p\Z)^\lambda$, so that $M$, which is isomorphic to  $K^r$ as a $K$-vector space, is an elementary  abelian $p$-group.

Our goal is to prove that all non-trivial elements of $(M,\circ)$ have order $p$.

Now,  $\delta_a^{p-1}(a)\in M^{p}$, and     by Lemma~\ref{lemma:dimension} $M^{p}=0$, so  Equation~\eqref{eq:p_pallino_a} reduces to 
$$p\circ a=(p +
      \binom{p}{2} \delta_a
      + \dots +
      \binom{p}{p-1} \delta_a^{p-2})(a),$$ 
      which  gives $p_\circ a=0$ for all $a\in M$.
\end{proof}
The next easy lemma is an immediate  generalisation of \cite[Lemma~2.7]{Bac16}.
 \begin{lemma}
 \label{lemma:2.7}
{\sl 
 Let $p$ be a prime and let $D$ be a PID in which $p$ is prime.  Let $N$ be a $D$-module of cardinality a power of $p$ and $D$-rank $r$. Then, for each $f\in \Aut_D(N)$ of order a power of $p$, $(f-id)^r(N)\subseteq pN$.
 }
 \end{lemma}
\begin{proof}

$pN$ is a characteristic subgroup  (and a $D$-submodule) of $N$, so, $f$ induces an automorphism $\bar f$ of $N/pN$. 
Since $N/pN$ is  a vector space over $K=D/pD$ of dimension $ r$, then the authomorphism$\bar f$ can be identified with a matrix $F\in{\rm GL}_r( K)$. As $F$ has order a power of $p$ and the characteristic of $K$ is $p$, then $F$ is conjugated to an upper triangular matrix $r\times r$ whose diagonal entries are 1. Therefore, $(F-Id)^r=0$, so $(f-id)^r(N)\subseteq pN.$
\end{proof}

For a  finite abelian $p$-group  $(G, \times)$, and  $i \ge 0$,  we will  denote by
$\Omega_{i}(G, \times)$  the set of elements of  $(G, \times)$  of order
dividing $p^{i}$. Clearly if $(G, \times)$ is abelian, then $\Omega_{i}(G, \times)$ a subgroup of $(G, \times)$.

The key point for proving the theorem is the following lemma
    \begin{lemma} 
    \label{lemma:prop5}
  Let everything be as in the statement   of Theorem ~\ref{teo:isomorphic}.
  Then,
  \begin{equation}
      \label{eq:a_tale_of_two_Omegas}
      \Omega_{i+1}(N, +) \setminus \Omega_{i}(N, +)
      \subseteq
      \Omega_{i+1}(N, \circ) \setminus \Omega_{i}(N,\circ),
    \end{equation}
      for each $i \ge 0$.
   \end{lemma}
  \begin{proof}
  First of all, we notice that in our case the subgroups $\Omega_i(N,+)$'s are characteristic in $(N,+)$ and therefore, they are also left $D$-ideals of the $D$-brace $(N,+,\circ)$, in particular each of them is a $D$-brace. 
  
  We start by proving that for each $i>0$
   \begin{equation}
      \label{eq:a_tale_of_one_Omega}
      \Omega_{i}(N, +)
      \subseteq
      \Omega_{i}(N, \circ).
    \end{equation}
   
We first consider the case $i=1$.  
 
As a $D$-module, $\Omega_1(N,+)$ is annihilated by $pD$, therefore it is a $K$-vector space of dimension $<p-1$, where $K=D/pD$ is a finite field (see Remark~\ref{rem:ff}). Proposition~\ref{prop:4} applies, giving $\Omega_{1}(N, +) \subseteq \Omega_{1}(N, \circ)$.

Let $i>0$ and assume by induction that
$$ \Omega_{i}(N, +) 
      \subseteq
      \Omega_{i}(N, \circ).
      $$
 Let $a\in\Omega_{i+1}(N, +)$. Then $pa^j\in \Omega_{i}(N, +)$ for all $j\ge1$, so from Equation~\eqref{eq:p_pallino_a} we get that $p_\circ a\equiv  \delta_a^{p-1}(a)=(\gamma_a-id)^{p-1}(a)\pmod{ \Omega_{i}(N, +)}$. 
 Since the rank of $\Omega_{i+1}(N, +)$ is $ < p-1$,  by Lemma~\ref{lemma:2.7}, we get $ \delta_a^{p-1}(a)\in p\Omega_{i+1}(N, +)\subseteq\Omega_{i}(N, +) $, so 
 $p_\circ a\in\Omega_{i}(N, +)\subseteq \Omega_{i}(N, \circ).$
Therefore, $a\in \Omega_{i+1}(N, \circ)$, proving the first statement. 
 
 We turn now to proving \eqref{eq:a_tale_of_two_Omegas}. Let $i>0$ and and assume by induction that
$$ \Omega_{i}(N, +) \setminus \Omega_{i-1}(N, +)
      \subseteq
      \Omega_{i}(N, \circ) \setminus \Omega_{i-1}(N,\circ).$$
      We have to show that
$$\Omega_{i+1}(N, +) \setminus \Omega_{i}(N, +)
      \subseteq
      \Omega_{i+1}(N, \circ) \setminus \Omega_{i}(N,\circ).
      $$
      Let $a\in\Omega_{i+1}(N, +) \setminus \Omega_{i}(N, +)$. We have already seen that  $a\in\Omega_{i+1}(N, \circ)$, so we are left to prove that $a$ is not in $\Omega_{i}(N, \circ)$. This can be reduced to showing that $p_\circ a\not\in \Omega_{i-1}(N,+)$. In fact, this implies $p_\circ a \in  \Omega_{i}(N, +) \setminus \Omega_{i-1}(N, +)
      \subseteq
      \Omega_{i}(N, \circ) \setminus \Omega_{i-1}(N,\circ),$ and this ensures that     $a$ is not in $\Omega_{i}(N, \circ)$. 
      
      To show that $p_\circ a$ is not in  $\Omega_{i-1}(N, +)$ we will prove that its class is not 0 in the quotient 
      $\Omega_{i+1}(N, +)/ \Omega_{i-1}(N, +).$ 
      
Denote by $\lowoverline{\delta_a^j( a)}$ the class of $\delta_a^j(a)$ in the $D$-module  $\Omega_{i+1}(N, +)/ \Omega_{i-1}(N, +)$. Define 
$S$ to be the $D$-submodule of $\Omega_{i+1}(N, +)/ \Omega_{i-1}(N, +)$ generated by  $\{\lowoverline{\delta_a^j( a)}\}_{j\ge0}$, and, for each $i\ge1$,
$$S_i=\langle  \lowoverline{\delta_a^j( a)}\mid j\ge i-1\rangle_D$$   
We get the strictly decreasing chain  
$$S=S_1\supset S_{2}\supset\dots\supset S_{t}\supset\{0\}$$
for some $t$. 
In fact, clearly $\delta_a(S_i)= S_{i+1}\subseteq S_i$; moreover, $S_n=0$ for $n$ sufficiently large, since using Lemma~\ref{lemma:2.7}, we get that if $p^lN=0$, then $\delta_a^{rl}(a)\in p^lN=0$, so $S_{rl+1}=0$. Finally, if $S_{i+1}=S_i$ then $S_{i+k}=S_i$ for all $k\ge0$, therefore this can only happen if $S_i=0$.

By Equation~\eqref{eq:p_pallino_a} 
$$p_\circ\bar a =
(p +
      \binom{p}{2} \delta_a
      + \dots +
      \binom{p}{p-1} \delta_a^{p-2})(\bar a)
+
  \delta_a^{p-1}(\bar a).
$$

If $p\bar a$ is not in $S_2$, then $p_\circ\bar a=p\bar a+s_2$, with $s_2\in S_2$, is not 0, so $p_\circ \bar a\not\in \Omega_{i-1}(N,+)$. Let $p\bar a\in S_2$.
Since  $\bar a$ has order $p^2$ , then $p\bar a$ is non-zero and   it belongs to $S_k\setminus S_{k+1}$ for some $k\ge2$.
This implies that the $D$-module $S/S_k$ is annihilated by $p$ so it is a $K$-module  and its dimension is $<p-1$.

On the other hand, using the notation $[s]$ for the class of $s$ in the quotient  $S/S_{k}$, we claim that $[\bar a], [\lowoverline{\delta_a( a)}], \dots,[\lowoverline{\delta_a^{k-2}(a)}]$ is a $K$-basis of  $S/S_k$. In fact, it is clear that they are a set of generators; if  they were linearly dependent then we would have  $\sum_{i=i_0}^{k-2}\lambda_i [\lowoverline{\delta_a^{i}(a)}]=0$ for some $\lambda_i\in K$ and $\lambda_{i_0}\ne 0$, so, by applying $\lowoverline{\delta_a^{k-i_0-2}}$ to the previous equality  we would get  $[\lowoverline{\delta_a^{k-2 }(a)}]=0$ and therefore $S_k=S_{k-1}$, a contradiction. 

This proves that the dimension over $K$ of $S/S_k$ is $k-1$, so  $k-1<p-1$, hence $ p-1\ge k$ and $\lowoverline{\delta_a^{p-1}(a)}\in S_{k+1}$. 

Finally, taking into account that $p\bar a\in S_k$ implies $p\delta_a^i(\bar a)\in S_{k+1}$ for all $i\ge1$, from the previous equation we get $p_\circ a\equiv pa\not\equiv0\pmod{S_{k+1}}$, therefore  it is not zero in the quotient $\Omega_{i+1}(N, +)/ \Omega_{i-1}(N, +)$. This proves that  $p_\circ a\not\in \Omega_{i-1}(N, +)$.
\end{proof}
\begin{proof}[Proof of Theorem~\ref{teo:isomorphic}]
Consider the  two partitions of $N$
\begin{equation*} 
  N
  =
  \ \cdot \hspace{-10.6pt}\bigcup_{i\ge0}\Omega_{i+1}(N, +)
  \setminus
  \Omega_{i}(N, +)
  =
  \ \cdot \hspace{-10.6pt}\bigcup_{i\ge0}\Omega_{i+1}(N, \circ)
  \setminus
  \Omega_{i}(N,\circ).
\end{equation*}
The last equality together
with~\eqref{eq:a_tale_of_two_Omegas} gives, for each $i \ge 0$,  
\begin{equation*}
  \Omega_{i+1}(N, +) \setminus \Omega_{i}(N, +)
  =
  \Omega_{i+1}(N, \circ) \setminus \Omega_{i}(N,\circ),
\end{equation*} 
showing that  in this case  the order of each  element is the  same in
$(N,  +)$ and  $(N,\circ)$. 

In particular, if   $(N, +)$ and $(N, \circ)$ are both abelian, they must be isomorphic.
\end{proof}

\bibliographystyle{amsalpha}
 
\bibliography{biblio}

\newcommand{\etalchar}[1]{$^{#1}$}
\def\Dbar{\leavevmode\lower.6ex\hbox to 0pt{\hskip-.23ex \accent"16\hss}D}
  \def\cfac#1{\ifmmode\setbox7\hbox{$\accent"5E#1$}\else
  \setbox7\hbox{\accent"5E#1}\penalty 10000\relax\fi\raise 1\ht7
  \hbox{\lower1.15ex\hbox to 1\wd7{\hss\accent"13\hss}}\penalty 10000
  \hskip-1\wd7\penalty 10000\box7}
  \def\cftil#1{\ifmmode\setbox7\hbox{$\accent"5E#1$}\else
  \setbox7\hbox{\accent"5E#1}\penalty 10000\relax\fi\raise 1\ht7
  \hbox{\lower1.15ex\hbox to 1\wd7{\hss\accent"7E\hss}}\penalty 10000
  \hskip-1\wd7\penalty 10000\box7}
\providecommand{\bysame}{\leavevmode\hbox to3em{\hrulefill}\thinspace}
\providecommand{\MR}{\relax\ifhmode\unskip\space\fi MR }
\providecommand{\MRhref}[2]{%
  \href{http://www.ams.org/mathscinet-getitem?mr=#1}{#2}
}
\providecommand{\href}[2]{#2}
\begin{thebibliography}{JKVAV19}

\bibitem[AB20a]{AB20-pq}
E.~Acri and M.~Bonatto, \emph{Skew braces of size {$pq$}}, Comm. Algebra
  \textbf{48} (2020), no.~5, 1872--1881. \MR{4085764}

\bibitem[AB20b]{AB20-sqrf}
A.~A. Alabdali and N.~P. Byott, \emph{Hopf--{G}alois structures of squarefree
  degree}, Journal of Algebra \textbf{559} (2020), 58--86.

\bibitem[AB22]{AB22}
E.~Acri and M.~Bonatto, \emph{Skew braces of size $p^2 q$ i: Abelian type},
  Algebra Colloquium \textbf{29} (2022), no.~02, 297--320.

\bibitem[Bac16]{Bac16}
D.~Bachiller, \emph{Counterexample to a conjecture about braces}, Journal of
  Algebra \textbf{453} (2016), 160--176.

\bibitem[Byo96]{Byo96}
N.~P. Byott, \emph{Uniqueness of {H}opf {G}alois structure for separable field
  extensions}, Comm. Algebra \textbf{24} (1996), no.~10, 3217--3228.
  \MR{1402555}

\bibitem[Byo04]{Byo04}
\bysame, \emph{Hopf--{G}alois structures on {G}alois field extensions of degree
  {$pq$}}, J. Pure Appl. Algebra \textbf{188} (2004), no.~1-3, 45--57.
  \MR{2030805}

\bibitem[Byo13]{Byo13}
\bysame, \emph{Nilpotent and abelian {H}opf--{G}alois structures on field
  extensions}, J. Algebra \textbf{381} (2013), 131--139. \MR{3030514}

\bibitem[Byo15]{Byo15}
\bysame, \emph{Solubility criteria for {H}opf-{G}alois structures}, New York J.
  Math. \textbf{21} (2015), 883--903. \MR{3425626}

\bibitem[Car20]{Car20}
A.~Caranti, \emph{Bi-skew braces and regular subgroups of the holomorph}, J.
  Algebra \textbf{562} (2020), 647--665. \MR{4130907}

\bibitem[CCDC20]{CCDC20}
E.~Campedel, A.~Caranti, and I.~Del~Corso, \emph{Hopf--{G}alois structures on
  extensions of degree {$p^2q$} and skew braces of order {$p^2 q$}: the cyclic
  {S}ylow {$p$}-subgroup case}, J. Algebra \textbf{556} (2020), 1165--1210.
  \MR{4089566}

\bibitem[CCS15]{CCS15}
F.~Catino, I.~Colazzo, and P.~Stefanelli, \emph{On regualr subgroups of the
  affine group}, Bulletin of the Australian Mathematical Society \textbf{91}
  (2015), no.~1, 76--85.

\bibitem[CCS19]{CCS19}
\bysame, \emph{Skew left braces with non-trivial annihilator}, Journal of
  Algebra and Its Applications \textbf{18} (2019), no.~02, 1950033.

\bibitem[CDC22]{CDC22}
A.~Caranti and I.~Del~Corso, \emph{On the ranks of the additive and the
  multiplicative groups of a brace}, Riv. Math. Univ. Parma (N.S.) \textbf{13}
  (2022), no.~1, 31--46. \MR{4456580}

\bibitem[CDV17]{CDV17}
A.~Caranti and F.~Dalla~Volta, \emph{The multiple holomorph of a finitely
  generated abelian group}, J. Algebra \textbf{481} (2017), 327--347.
  \MR{3639478}

\bibitem[CDV18]{CDV18}
\bysame, \emph{Groups that have the same holomorph as a finite perfect group},
  J. Algebra \textbf{507} (2018), 81--102. \MR{3807043}

\bibitem[CDVS06]{CDVS06}
A.~Caranti, F.~Dalla~Volta, and M.~Sala, \emph{Abelian regular subgroups of the
  affine group and radical rings}, Publ. Math. Debrecen \textbf{69} (2006),
  no.~3, 297--308. \MR{2273982}

\bibitem[CGK{\etalchar{+}}21]{CGKKKTU}
L.~N. Childs, C.~Greither, K.~P. Keating, A.~Koch, T.~Kohl, P.~J. Truman, and
  R.~G. Underwood, \emph{Hopf algebras and {G}alois module theory},
  Mathematical Surveys and Monographs, vol. 260, American Mathematical Society,
  Providence, RI, [2021] \copyright 2021. \MR{4390798}

\bibitem[Chi89]{Chi89}
L.~N. Childs, \emph{On the {H}opf--{G}alois theory for separable field
  extensions}, Comm. Algebra \textbf{17} (1989), no.~4, 809--825. \MR{990979}

\bibitem[Chi05]{Chi05}
\bysame, \emph{Elementary abelian {H}opf {G}alois structures and polynomial
  formal groups}, J. Algebra \textbf{283} (2005), no.~1, 292--316. \MR{2102084}

\bibitem[Chi19]{Chi19}
\bysame, \emph{Bi-skew braces and {H}opf--{G}alois structures}, New York J.
  Math. \textbf{25} (2019), 574--588. \MR{3982254}

\bibitem[CJO14]{CJO14}
F.~Cedo, E.~Jespers, and J.~Okni{\'n}ski, \emph{Braces and the {Y}ang--{B}axter
  equation}, Communications in Mathematical Physics \textbf{327} (2014),
  101--116.

\bibitem[CSV19]{CSV19}
F.~Ced{\'o}, A.~Smoktunowicz, and L.~Vendramin, \emph{Skew left braces of
  nilpotent type}, Proceedings of the London Mathematical Society \textbf{118}
  (2019), no.~6, 1367--1392.

\bibitem[DCD18a]{dcdAMPA}
I.~Del~Corso and R.~Dvornicich, \emph{Finite groups of units of finite
  characteristic rings}, Annali di Matematica \textbf{197} (2018), 66--671.

\bibitem[DCD18b]{dcdBLMS}
\bysame, \emph{On {F}uchs' {P}roblem about the group of units of a ring}, Bull.
  London Math. Soc. (2018), no.~50, 274--292.

\bibitem[FCC12]{FCC}
S.~C. Featherstonhaugh, A.~Caranti, and L.~N. Childs, \emph{Abelian
  {H}opf--{G}alois structures on prime-power {G}alois field extensions}, Trans.
  Amer. Math. Soc. \textbf{364} (2012), no.~7, 3675--3684. \MR{2901229}

\bibitem[Fuc60]{Fuchs60}
L.~Fuchs, \emph{Abelian groups}, 3rd ed., Pergamon, Oxford, 1960.

\bibitem[GP87]{GP87}
C.~Greither and B.~Pareigis, \emph{Hopf {G}alois theory for separable field
  extensions}, J. Algebra \textbf{106} (1987), no.~1, 239--258. \MR{878476}

\bibitem[GV17]{GV17}
L.~Guarnieri and L.~Vendramin, \emph{Skew braces and the {Y}ang--{B}axter
  equation}, Math. Comp. \textbf{86} (2017), no.~307, 2519--2534. \MR{3647970}

\bibitem[JKVAV19]{JKVAV19}
E.~Jespers, {\L}.~Kubat, A.~Van~Antwerpen, and L.~Vendramin,
  \emph{Factorizations of skew braces}, Mathematische Annalen \textbf{375}
  (2019), no.~3, 1649--1663.

\bibitem[Koh98]{koh98}
T.~Kohl, \emph{Classification of the {H}opf {G}alois structures on prime power
  radical extensions}, J. Algebra \textbf{207} (1998), no.~2, 525--546.
  \MR{1644203}

\bibitem[Lan94]{Lang94ANT}
S.~Lang, \emph{Algebraic number theory}, 2nd ed., Graduate Texts in
  Mathematics, vol. 110, Springer-Verlag, New York, 1994.

\bibitem[Nas19]{Nas19}
T.~Nasybullov, \emph{Connections between properties of the additive and the
  multiplicative groups of a two-sided skew brace}, J. Algebra \textbf{540}
  (2019), 156--167. \MR{4003478}

\bibitem[NZ18]{zen18}
K.~Nejabati~Zenouz, \emph{{O}n {H}opf--{G}alois {S}tructures and {S}kew
  {B}races of {O}rder $p^3$}, PhD thesis, The University of Exeter (2018),
  \url{https://ore.exeter.ac.uk/repository/handle/10871/32248}.

\bibitem[Rum06]{Rum06}
W.~Rump, \emph{Modules over braces}, Algebra Discrete Math. (2006), no.~2,
  127--137. \MR{2320986}

\bibitem[Rum07]{Rum07}
\bysame, \emph{Braces, radical rings, and the quantum {Y}ang-{B}axter
  equation}, J. Algebra \textbf{307} (2007), no.~1, 153--170. \MR{2278047}

\bibitem[Ser95]{Ser95}
J.P. Serre, \emph{Local fields}, Graduate Texts in Mathematics, Springer New
  York, 1995.

\bibitem[Smo22a]{Smo22a}
A.~Smoktunowicz, \emph{Algebraic approach to {R}ump's results on relations
  between braces and pre-{L}ie algebras}, J. Algebra Appl. \textbf{21} (2022),
  no.~3, Paper No. 2250054, 13. \MR{4391819}

\bibitem[Smo22b]{Smo22b}
\bysame, \emph{A new formula for {L}azard's correspondence for finite braces
  and pre-{L}ie algebras}, J. Algebra \textbf{594} (2022), 202--229.
  \MR{4353236}

\bibitem[ST22a]{ST22a}
L.~Stefanello and S.~Trappeniers, \emph{On bi-skew braces and brace blocks},
  2022.

\bibitem[ST22b]{ST22b}
\bysame, \emph{On the connection between {Hopf}--{Galois} structures and skew
  braces}, 2022.

\bibitem[SV18]{SV18}
A.~Smoktunowicz and L.~Vendramin, \emph{On skew braces (with an appendix by
  {N}. {B}yott and {L}. {V}endramin)}, J. Comb. Algebra \textbf{2} (2018),
  no.~1, 47--86. \MR{3763907}

\bibitem[TQ20]{TQ20}
C.~Tsang and C.~Qin, \emph{On the solvability of regular subgroups in the
  holomorph of a finite solvable group}, Internat. J. Algebra Comput.
  \textbf{30} (2020), no.~2, 253--265. \MR{4077413}

\bibitem[Tsa19]{Tsa19}
C.~Tsang, \emph{Non-existence of {H}opf--{G}alois structures and bijective
  crossed homomorphisms}, J. Pure Appl. Algebra \textbf{223} (2019), no.~7,
  2804--2821. \MR{3912948}

\end{thebibliography}

\end{document}